\documentclass[final, abstracton, paper=a4, fontsize=11pt,DIV=12,bibliography=totoc]{scrartcl}

\pdfoutput=1

\usepackage[utf8]{inputenc}
\usepackage[T1]{fontenc}
\usepackage{lmodern}
\usepackage[english]{babel}
\usepackage{csquotes}
\usepackage{amsthm, amssymb, amsmath, amsfonts, mathrsfs, dsfont, esint,textcomp}
\usepackage[colorlinks=true, pdfstartview=FitV, linkcolor=blue, citecolor=blue, urlcolor=blue,pagebackref=false]{hyperref}
\usepackage{mathtools}
\usepackage[shortlabels]{enumitem}

\usepackage[backend=biber,giveninits,maxnames=4,maxalphanames=5, style=alphabetic,isbn=false, date=year, doi=false, eprint= true, url= false, sorting=ynt, sortcites = true]{biblatex}

 \usepackage{authblk}

 \setkomafont{date}{\large}

 \addtokomafont{disposition}{\rmfamily}

\usepackage{microtype}

\usepackage[notcite,notref,color]{showkeys}
\definecolor{labelkey}{gray}{.8}
\definecolor{refkey}{gray}{.8}

\definecolor{darkblue}{rgb}{0,0,0.7} 
\definecolor{darkred}{rgb}{0.9,0.1,0.1}
\definecolor{darkgreen}{rgb}{0,0.5,0}

\renewenvironment{proof}[1][\smallskip\noindent\proofname]{{\smallskip\noindent\bfseries #1. }}{\qed \medskip}

\newtheorem{thm}{Theorem}[section]

\newtheorem{prop}[thm]{Proposition}
\newtheorem{lem}[thm]{Lemma}

\theoremstyle{definition}

\newtheorem{rem}[thm]{Remark}

\newcommand\norm[1]{\left\lVert#1\right\rVert}

\newcommand\bra[1]{\left({#1}\right)}

\newcommand\abs[1]{\left\lvert#1\right\rvert}

\renewcommand{\le}{\leqslant}
\renewcommand{\ge}{\geqslant}
\renewcommand{\leq}{\leqslant}
\renewcommand{\geq}{\geqslant}
\newcommand{\ls}{\lesssim}

\newcommand{\E}{\mathbb{E}}

\newcommand{\J}{\mathcal{J}}

\newcommand{\mP}{\mathcal{P}}

\newcommand{\N}{\mathbb{N}}

\newcommand{\1}{\mathbf{1}}
\newcommand{\R}{\mathbb{R}}

\renewcommand{\P}{\mathbb{P}}

\newcommand{\eps}{\varepsilon}

\newcommand{\dd}{\, \mathrm{d}}

\newcommand{\W}{\mathcal{W}}
\newcommand{\dmin}{d_{\mathrm{min}}}
\newcommand{\dminone}{d_{\mathrm{min,1}}}

\renewcommand{\varrho}{\rho}

\DeclareMathOperator*{\esssup}{esssup}

\DeclareMathOperator{\Id}{Id}

\DeclareMathOperator{\dv}{div}

\numberwithin{equation}{section}

\mathtoolsset{showonlyrefs}

\bibliography{bib.bib}

\begin{document}


\title{\LARGE Propagation of Chaos for First-Order Mean-Field Systems with Non-Attractive Moderately Singular Interaction}
\author[1]{Richard M. H\"ofer\thanks{richard.hoefer@ur.de}}
\author[2]{Richard Schubert\thanks{schubert@iam.uni-bonn.de}}
\affil[1]{Faculty of Mathematics, University of Regensburg, Germany}
\affil[2]{Institute for Applied Mathematics, University of Bonn, Germany}


\maketitle

\begin{abstract}
    We consider particle systems that evolve by inertialess binary interaction through general non-attractive kernels  of singularity $|x|^{-\alpha}$ with $\alpha<d-1$. We prove a quantitative mean-field limit in terms of Wasserstein distances under certain conditions on the initial configuration while maintaining control of the particle configuration in the form of the minimal distance and certain singular sums of the particle distances. As a corollary, we show propagation of chaos for $\alpha<\frac{d-1}{2}$ for $d\ge 3$ and $\alpha<\frac 13=\frac{2d-3}{3}$ for $d=2$. This  extends the results of Hauray \cite{Hauray09}, which  yield propagation of chaos for $\alpha < \frac{d-2}{2}$ without an assumption on the sign of the interaction.   
    The main novel ingredient is that due to the non-attraction property it is enough to control the distance to the next-to-nearest neighbour particle.
\end{abstract}

\section{Introduction}

We consider the classical first-order $N$ particle mean-field dynamics in $\R^d$ given by
\begin{align}\label{eq:micro}
    \left\{\begin{array}{rl} \displaystyle\frac{\dd}{\dd t} X_i &= \displaystyle\frac 1 N \displaystyle \sum_{j \neq i}K(X_i -X_j), \\
    X_i(0) &= X_i^0,\end{array}\right.
\end{align}
for some interaction kernel $K$ and given initial positions $X_i^0\in \R^d, i=1,\dots,N$. A typical question in this context is the following: Denote the empirical measure $\rho_N (t) = \frac 1 N \sum_i \delta_{X_i(t)}$, and assume that $\rho_N(0)\rightharpoonup \rho^0$ in an appropriate sense, is it true that $\rho_N(t)\rightharpoonup \rho(t)$ for $t>0$, and if so, which PDE does $\rho$ satisfy? This PDE is called the mean-field limit of \eqref{eq:micro}.

The expected mean-field limit for \eqref{eq:micro} is given by the equation
\begin{align}\label{eq:macro}
\left\{\begin{array}{rl}\partial_t \rho+\dv((K\ast \rho)\rho)&=0,\\
\rho(0)&=\rho^0.
\end{array}\right.
\end{align}
{Indeed, the empirical measure $\rho_N$ is already a distributional solution to this PDE (with the convention $K(0) = 0$).}
A typical assumption on the initial particle positions is that they are \emph{chaotic}, i.e. independently and identically distributed (i.i.d.)  with law $\rho^0$. The notion of \emph{propagation of chaos} is then the convergence to the mean-field limit \eqref{eq:macro}  with probability $1$ asymptotically as $N \to \infty$. We refer to \cite{JabinReview, ChaintronDiezReviewA} for review articles on this topic.\\

For Lipschitz interaction kernels $K$, Dobroushin \cite{Dobroushin79} showed a stability result for distributional solutions to \eqref{eq:macro} in terms of the $1$-Wasserstein distance, which immediately implies propagation of chaos for such kernels.
Usually the main obstacle in deriving a mean-field limit for \eqref{eq:micro} arises from the singularity of the kernel at the origin $|K(x)|\sim |x|^{-\alpha}$. There are numerous results that deal with this issue by introducing an $N$-dependent truncation of this singularity. For conciseness, we only mention here some  of the results that consider first-order mean-field limits with the non-truncated singular kernel. \\

Work on this topic started with the so-called point-vortex system in $d=2$ where the specific kernel is $K(x)=-x/|x|^2$ and the mean-field limit is Euler's equation. The convergence was obtained in \cite{Sch96}. Motivated by the point-vortex system, but considering general singular kernels, Hauray's original method \cite{Hauray09} builds on Dobroushins approach, but  maintainins simultaneous control on the $\infty$-Wasserstein distance between the empirical density to the limit measure and the minimal inter-particle distance (which implicitly depends on $N$)
\begin{align}\label{eq:dmin}
     \dmin(t) = \min_i \min_{j \neq i} |X_i(t) - X_j(t)|.
\end{align}
This allows to show mean-field limits for $\alpha<d-1$ in any dimension $d$ under certain assumptions on the initial Wasserstein distance and the initial minimal inter-particle distance.  Moreover, Hauray's result  implies \emph{propagation of chaos} for $\alpha < \frac {d-2}2$ because aforementioned assumptions are satisfied for such $\alpha$ with overwhelming probability for i.i.d. initial particle positions as $N \to \infty$.

Through a modulated energy argument, Serfaty showed in the seminal work \cite{Ser20} the propagation of chaos for very singular interaction kernels in the special repulsive form $K(x)=x/|x|^{\alpha+1}$ for $d-1 \leq \alpha<d+1$. Notably, this includes the Coulomb interaction kernel.
Based on this method and relative entropy methods from \cite{JW16, JW18}, this result has been improved to include more general (repulsive, gradient-type) interaction kernels, less singular attractive potentials and  added diffusive noise,  see for instance \cite{NguyenRosenzweigSerfaty22, BJW19, DeCourcelRosenzeigSerfaty23, BreschJabinWang23}

Very recently, by a novel duality approach, \cite{BDJ24} obtained propagation of chaos for interaction kernels with $\alpha<\frac{d+2}2$,  under the symmetry condition $K(x)=-K(-x)$.\\

A particular benefit of Hauray's method is that it propagates control of the minimal particle distance. This makes the method sufficiently robust to treat certain mean-field problems with non-binary interaction. A prime example for this kind of interaction is sedimentation.
The intricate gravity-driven dynamics of particles suspended in a Stokes fluid is, to leading order, given by the binary interaction through the interaction kernel $K=\Phi g$ where $g$ is the constant vector of gravitational acceleration and $\Phi$ is the Stokes kernel (Oseen tensor)
\begin{align} \label{Oseen}
    \Phi(x)=\frac 1{8\pi|x|}\bra{\Id+\frac{x\otimes x}{|x|^2}}.
\end{align}
Deriving mean-field limits for such systems, starting from a coupled fluid-particle system, builds on the combination of two ingredients. First, a good approximation of the implicit interaction by the explicit one, and second, a mean-field limit for the explicit system \eqref{eq:micro}. In order to be able to do the first approximation step, it is essential to have a good control on the minimal distance between the particles during the evolution. Therefore, despite the tremendous progress in mean-field theory in recent years, all the
results on mean-field limits for sedimentation  \cite{Hofer18MeanField, Mecherbet19, Hofer&Schubert, HoferSchubert23a, Duerinckx23} are built on Hauray's method. In these results, Hauray's method has been improved, firstly to allow to replace the $\infty$-Wasserstein distance by $p$-distances, secondly to include error terms (that arise from the binary approximation), and thirdly regarding the assumptions on the Wasserstein distance and minimal inter-particle   distance. However, these improvements have not allowed to push the threshold $\alpha < \frac {d-2}2$ for propagation of chaos, provided by this method. Since  for sedimentation one has $\alpha = 1$ and $d =3$, the results in \cite{Hofer18MeanField, Mecherbet19, Hofer&Schubert, HoferSchubert23a,Duerinckx23} therefore only hold for sufficiently well prepared initial particle configurations, leaving propagation of chaos as an important open problem.
On the level of the minimal inter-particle distance, so far Hauray's method requires $\dmin \gg N^{-1/2}$ for $\alpha = 1$ and $d=3$
whereas chaotic initial data typically have $\dmin \sim N^{-2/3}$. More generally, the threshold in Hauray's method for the minimal inter-particle distance is $\dmin \gg N^{-1/(1+\alpha)}$ (for any space dimension $d$).\\



In the present article we show that for non-attractive kernels (an important example being the Stokes kernel), the assumptions on the inter-particle distance and the Wasserstein distance in Hauray's method can be relaxed. Our strategy exploits the non-attractive property in order to essentially replace $\dmin$ by the minimal distance to the next-to-nearest particle $\dminone$ both as a quantity to be monitored during the evolution, and in the assumptions on the initial configuration. This leads to a  replacement of the threshold $\dmin \gg N^{-1/(1+\alpha)}$ by  $\dminone \gg N^{-1/(1+\alpha)}$. Since $\dminone$  behaves much better than $\dmin$ for chaotic initial data ($\dmin \sim N^{-2/d}, \dminone \sim N^{-3/(2d)}$ with overwhelming probability), we deduce propagation of chaos under the less stringent condition $\alpha < \min\{ \frac{d-1}2,\frac{2d-3}{3}\} $ (compared to the previous threshold $\alpha < \frac {d-2}2$).
Unfortunately, the case of sedimentation, our guiding example, is critical in the sense that $\alpha = 1 = \frac {d-1}2 = \frac{2d-3}{3}$ in dimension $d = 3$, and, statistically, $\dminone \sim N^{-1/2} = N^{-1/(1+\alpha)}$.
Therefore, our methods currently do not seem to allow to treat propagation of chaos for sedimentation, but they can be used to considerably relax the assumptions on the initial particle positions in \cite{Hofer18MeanField, Mecherbet19, Hofer&Schubert, HoferSchubert23a, Duerinckx23}.


\subsection{Notation}

\begin{itemize}
 \item For $ X\in  \bra{\R^d}^N$ we write 
    \begin{align}
        \dmin &\coloneqq \min_i \min_{j \neq i} |X_i - X_j|, \\
        d_{ij} &\coloneqq \begin{cases}
                |X_i - X_j|, & \quad  \text{for} ~ i \neq j, \\
                \dmin & \quad  \text{for} ~ i = j.
                \end{cases}
    \end{align}
   For a given index $i$, there is an index $i_{nn}\neq i$ such that $d_{ii_{nn}}\le d_{ij}$ for all $j\neq i$ (a nearest neighbour). If there are several nearest neighbours we choose one to be $i_{nn}$. This implies that $i_{nn}$ may be time-dependent. We write
   \begin{align}
        \dminone \coloneqq \min_i\min_{j\neq i,i_{nn}} d_{ij}. \\
    \end{align}
    It may be that $\dmin=\dminone$ for a configuration with two nearest neighbours. We will usually not explicitly denote the $N$-dependence for quantities like $\dmin,\dminone$. However we will always denote by $\rho_N$ the empirical density corresponding to $X$.
    
    \item For $\delta > 0$, we introduce the set of particles with close nearest neighbours, i.e. for $\delta > 0$, we denote
    \begin{align}
        D_\delta(X) := \{ X_i : d_{i,i_{nn}} < \delta\},
    \end{align}
    and we will typically omit the dependence on $X$.

    \item The pushforward of a measure $\rho$ by a measurable map $\Phi$ is defined by $\Phi_\#\rho(A)=\rho(\Phi^{-1}(A))$ for all measurable sets $A$. For probability measures $\rho^1,\rho^2\in \mP(\R^d)$ we denote by $\Gamma(\rho^1,\rho^2)$ the set of couplings, i.e. probability measures $\gamma\in\mP(\R^d\times \R^d)$ such that ${\pi_i}_\#\gamma=\rho^i$ for $i=1,2$ with $\pi_i$ being the projection to the respective component. For $p\in [1,\infty)$ we denote the $p$-Wasserstein distance by
\begin{align*}
    \W_p(\rho^1,\rho^2)\coloneqq \inf_{\gamma\in \Gamma(\rho^1,\rho^2)}\left(\int_{\R^3\times \R^3}|x-y|^p\dd \gamma(x,y)\right)^\frac 1p,
\end{align*}
while the $W_\infty$ distance is given by 
\begin{align*}
    \W_\infty(\rho^1,\rho^2)\coloneqq \inf_{\gamma\in \Gamma(\rho^1,\rho^2)}\gamma-\!\!\!\!\esssup_{(x,y)\in \R^d\times \R^d} |x-y|.
\end{align*}
    \item We will consider measures that have a density with respect to the Lebesgue measure and will write $\rho\in \mP(\R^d)\cap L^\infty(\R^d)$ when $\rho$ is a probability measure and its density (which we also call $\rho$) is in $L^\infty(\R^3)$.
    
    \item We abbreviate  $\|\cdot\|_{s}=\|\cdot\|_{L^s(\R^d)}$.
    
    \item We write $A\ls B$ to mean $A\le CB$ for some global constant $C<\infty$ that does not depend on the specific particle configuration or $N$. However,
we will sometimes allow $C$ to depend on the initial density for the macroscopic system $\rho^0$ when indicated in the statements. We use $A\ll B$ to imply that the quotient $A/B\to 0$ as $N\to \infty$.

\end{itemize}

\subsection{Statement of the main results}

Our first main result is a deterministic statement, that establishes a mean-field limit under the assumption that the kernel is non-attractive and that the initial distances of the particles scale according to the singularity of the kernel. Let $0< \alpha < d-1$. We assume the following on the kernel $K:\R^d\to \R^d$.
\begin{align} \label{eq:C_alpha} \tag{$C_\alpha$}
	\forall x \in \R^d&: |K(x)| + |x| |\nabla K(x)| \leq \frac{C_K}{|x|^{\alpha}}, ~ \dv K = 0, \\\forall x \in \R^d&:
 (K(x)-K(-x))\cdot x\ge 0. \label{ass:nonattractive}
 \tag{NA}
\end{align}

For $\rho^0\in \mP(\R^d)\cap L^\infty(\R^d)$ it is classical, that \eqref{eq:macro} has a unique global distributional solution since $K \ast \rho$ is Lipschitz (see \cite{Hauray09}).

\begin{thm} \label{th:deterministic.new}
    Let $0< \alpha < d-1$. Assume \eqref{eq:C_alpha} and \eqref{ass:nonattractive}. Let $T>0$, $\rho^0\in \mP(\R^d)\cap L^\infty(\R^d)$, $p \in [1,\infty]$. For each $N \in \N$, consider initial particle positions $X_i^0$, $1 \leq i \leq N$ such that  there exists  $ \delta_N \leq \dminone(0)$ with 
    \begin{align}
        \W_p(\rho_N^0,\rho^0)\to 0&, \label{ass:conv}\\
              \frac{1}{N^{(\alpha+1)/d} \delta_N^{\alpha+1}}\left(\W_p(\rho_N^0,\rho^0)^{\frac{(d-\alpha-1)p}{d+p}} +\left(N^{-p} \rho_N^0(D_{\delta_N}) \dmin(0)^{-\alpha p}\right)^{\frac{(d-\alpha-1)}{d+p}}\right)\to 0  \label{ass:W_p}&,  \\[3pt]
     \forall i \neq j \neq k \neq i: \quad  ( d_{ij}(0) \leq \delta_N \Longrightarrow d_{ik}(0) \gg \delta_N )&,  \label{ass:dminone.strong.1} \\[3pt]
             \forall i \neq j \neq k \neq i: \quad  ( d_{ik}(0) < \delta_N \Longrightarrow \frac 1 N d_{ij}(0)^{-1} d_{ik}(0)^{-\alpha} \ll 1 )&  \label{ass:dminone.strong.2} 
    \end{align}
    for all $N \in \N$.
    Let $\rho(t),\rho_N(t)$ be solutions of the corresponding equations \eqref{eq:macro} and \eqref{eq:micro} respectively. Then, there exist $C>0$ only depending on $d,\rho^0,p,C_K$, and $N_0 > 0$, depending in addition on $T$,  such that for all $N \geq N_0$ and all $t\in [0,T]$
    \begin{align}
        \W_p(\rho_N(t),\rho(t))
        &\le C\left(\W_p(\rho_N^0,\rho^0)+N^{-1}(\rho_N^0(D_{\delta_N}))^{1/p} \dmin(0)^{-\alpha}\right)e^{Ct}, \label{est:W.thm}\\
        \forall i\neq j: d_{ij}(t)&\ge  e^{-Ct}d_{ij}(0).
    \end{align}
\end{thm}

\begin{rem}\label{rem:det}
\begin{enumerate}[(i)]
  \item In the case $p=\infty$, condition \eqref{ass:W_p} has to be understood in the sense
  \begin{align*}
      \frac{1}{N^{(\alpha+1)/d} \delta_N^{\alpha+1}}\left(\W_\infty(\rho_N^0,\rho^0)^{d-\alpha-1}+(N^{-1}\dmin(0)^{-\alpha})^{d-\alpha-1} \right)\to 0.
  \end{align*}
  \item Note that the threshold $\alpha<d-1$ corresponds to integrability of $\nabla K$ at the origin. The lower bound $\alpha>0$ is for convenience to control the behaviour of $K$ at infinity. It is not relevant for compactly supported $\rho^0$. It is possible to relax the structural assumption $\dv K = 0$ at the cost of losing global in
  	time existence of the limiting equation (cf.  \cite{CarrilloChoiHauray14}).
  \item \eqref{ass:conv} is not necessary. It can be removed by replacing $\dmin(0)^{-\alpha p}$ by $(\dmin(0)+N^{-1/d})^{-\alpha p}$ in \eqref{ass:W_p}. We prefer however to make the quite natural assumption \eqref{ass:conv} in order to be more concise.
  \item\label{it:Wass_scaling} Note that the second term on the right-hand side of \eqref{est:W.thm} can be absorbed into the first one if it is bounded by $N^{-1/d}$ since $\W_p(\rho_N(t),\rho(t)) \gtrsim N^{-1/d}$ (see \cite[equation (1.11)]{Hofer&Schubert}, \cite[equation (1.14)]{HoferSchubert23a}).
  \item As mentioned in the introduction, this result allows in particular to relax the assumptions on the initial configurations in the derivation of the transport-Stokes equation as mean-field limit for sedimenting particles (cf. \cite{Hofer18MeanField,Mecherbet19,Hofer&Schubert,HoferSchubert23a}) in $\R^3$. Combining Theorem~\ref{th:deterministic.new} with the method of reflections to approximate the true fluid velocity, it is possible to allow for configurations that are close to random in the sense that they merely satisfy $\dmin\gg N^{-2/3}$, $\dminone\gg N^{-1/2}$, and with statistics as in \eqref{ass:dminone.strong.1}--\eqref{ass:dminone.strong.2},  whereas before the threshold was $\dmin\gg N^{-1/2}$ when $\W_p(\rho_N,\rho)\sim N^{-1/3}$, which is contained in the present result. 
\end{enumerate}
\end{rem}

For our second main result we observe that  particles with i.i.d. initial positions already satisfy the above initial conditions with overwhelming probability if the singularity of the kernel is not too strong, such that we get the following probabilistic statement.

\begin{thm}\label{th:main}
    Let $d\ge 2$ and $0 < \alpha< \min\{ \frac{2d-3}3, \frac{d-1}{2}\}$. Let $\rho^0\in \mP(\R^d)\cap L^\infty(\R^d)$ be the pushforward by a Lipschitz map of the uniform measure on the cube $\1_Q$. Let $T>0$. Let the family $\rho_N^0$ be empirical measures of $N$ i.i.d. particles  with law $\rho_0$. Let $\rho(t),\rho_N(t)$ be solutions of the corresponding equations \eqref{eq:macro} and \eqref{eq:micro} respectively. Then, for $p \in [1,\infty]$ with 
   \begin{align*}
        p>\frac{d(\alpha+1)}{2d-3(\alpha+1)},
    \end{align*}
    there exists $C>0$, only depending on $d,\rho^0,p$, and $C_K$, such that 
    \begin{align}
        \lim_{N\to \infty}\P\left(\forall t\in[0,T]:~\W_p(\rho_N(t),\rho(t))\le C \W_p(\rho_N^0,\rho^0)e^{Ct}\right)&=1,\qquad\label{eq:prob1}\\
        \lim_{N\to \infty}\P\left(\forall t\in [0,T], i\neq j:~d_{ij}(t)\ge e^{-Ct} d_{ij}(0)\right)&=1.
    \end{align}
\end{thm}

\begin{rem}
\begin{enumerate}[(i)]
\item The condition that $\rho^0$ should be the Lipschitz-pushforward of the unit distribution on the cube can probably be relaxed and is owed solely to the state of the art in estimating the tail behaviour for the scaling of the Wasserstein distance for randomly seeded particles (see Lemma \ref{lem:Wass_scaling} below). The theorem will work for any $\rho^0\in \mP(\R^d)\cap L^\infty(\R^d)$ for which the Wasserstein distance $\W_p(\rho_N,\rho)$ scales (up to logarithms) like $N^{-1/d}$ with asymptotic probability 1.  
\item Notice that the constraint on $p$ degenerates as $\alpha\to \frac {2d-3}3$ which is smaller than $\frac{d-1}2$ in $d=2$, coincides with $\frac{d-1}2$ in $d=3$, and is larger than $\frac{d-1}2$ for $d\ge 4$. The reason for the bound of $\alpha$ lies in condition \eqref{ass:dminone.strong.2}, which in turn is needed to maintain control of the interparticle distances. Here, an interesting change in scaling takes place in dimension $d=3$: As established in Lemma \ref{lem:three_particle1} below, for triples of particles to appear with non-vanishing probability, the product of the smallest and largest appearing distance must be at least of order $N^{-3/d}$, while the smallest appearing distance in pairs of particles is of order $N^{-2/d}$. Thus, while for higher dimensions ($d\ge 3$, where the threshold for $\alpha$ is $\ge 1$), the critical triples of particles in condition \eqref{ass:dminone.strong.2} are the ones where two particles are very close (with distance of order $N^{-2/d}$) and one is at distance $N^{-1/d}$, the critical triples for $d=2$ are the ones with mutually comparable spacing of order $N^{-3/(2d)}$. 
  \item This result considerably relaxes the assumption on the singularity of non-attractive interaction kernels known so far. Indeed, in \cite{Hauray09} the threshold for such a probabilistic result is $\alpha < \frac {d-2}2$ (see \cite[Remark~2.3]{Hauray09}). The more recent result \cite[Theorem~2.1]{HoferSchubert23a}  does not improve this threshold. 
  \item Note that $\alpha=1$ is critical in $d=3$. This is exactly the singularity of the Oseen-tensor \eqref{Oseen}, which shows that the problem of sedimentation is critical for the treatment of i.i.d. particles.
\end{enumerate}
\end{rem}

We prove Theorem~\ref{th:deterministic.new} in Section~\ref{sec:proof_main} and Theorem~\ref{th:main} in Section~\ref{sec:prob}.

\section{Proof of the deterministic result}\label{sec:proof_main}

\subsection{Control on the sum of inverse distances}

We define the cut-off sum
\begin{align} \label{S_beta.delta}
    S_{\beta,\delta}=\sup_i\sum_{\{j:d_{ij}>\delta\}}d_{ij}^{-\beta}.
\end{align}
 The following result is a straight-forward adaption of Lemma~\cite[Lemma 2.3]{HoferSchubert23a}. We provide the proof for the convenience of the reader.

\begin{lem} \label{lem:sums.Wasserstein}
Let $\beta \in (0,d)$, $p \in [1,\infty]$. Furthermore let $X \in (\R^d)^N$, and $\sigma_N = \frac 1 N \sum \delta_{X_i}$, as well as  $\sigma \in \mP(\R^d)\cap L^\infty(\R^d) $. Let $M,\delta>0$ be such that 
\begin{align}
    \sup_i\#\{j:X_j\in B_{\delta}(X_i)\}\le M. 
\end{align}
Then, $S_{\beta,\delta}$ is estimated by
\begin{align}
 \label{eq:sums.Wasserstein.p}
	\frac 1NS_{\beta,\delta}\lesssim \|\sigma\|_{\infty}^{\frac{\beta}d} +\left(\frac{M^{\beta/d}\|\sigma\|_{\infty}^{\frac{d-\beta}d \frac {p} {d+p }}}{N^{\beta/d} \delta^{\beta}}+\left(\frac{M^{\beta/d}\|\sigma\|_{\infty}^{\frac{d-\beta}d }}{N^{\beta/d} \delta^{\beta}}\right)^{\frac{\beta+p}{d+p}}\right) (\W_p(\sigma_N,\sigma))^{\frac{(d-\beta)p}{d+p}}.
\end{align}

In the case $p = \infty$, this should be understood as
\begin{align} \label{eq:sums.Wasserstein}
	\frac 1NS_{\beta,\delta}
	\lesssim \|\sigma\|_{\infty}^{\frac{\beta}{d}} +  M^{\beta/d}\|\sigma\|_{\infty}^{\frac{d-\beta}d}N^{-\beta/d}\delta^{-\beta} (\W_\infty(\sigma_N,\sigma))^{d-\beta}.
\end{align}
\end{lem}

Both in the proof of Lemma~\ref{lem:sums.Wasserstein} and further below, we will make use of the following fact. Assume $\sigma^i \in \mP(\R^d), i=1,2$ with $\W_p(\sigma^1,\sigma^2)<\infty$. Then it is classical that there exists an optimal transport plan $\gamma\in \Gamma(\sigma^1,\sigma^2)$ such that 
\begin{align}
		\W_p(\sigma^1,\sigma^2)= \bra{\int_{\R^d\times \R^d
    } |x-y|^p \dd \gamma(x,y)}^{1/p}, \quad \text{ for } p\in [1,\infty)
\end{align} 
and analogously for $p=\infty$.
\begin{proof}
We only consider $p< \infty$ here, the case $p = \infty$ can be obtained by an analogous adaption of the proof of \cite[Lemma~3.1]{Hofer&Schubert}. We fix $1 \leq i \leq N$ and adopt the short notation $\eta \coloneqq \mathcal W_p(\sigma_N,\sigma) $.\\

\noindent\textbf{Step 1:} 
\emph{Splitting of the sum.} 
If $\W_p(\sigma_N,\sigma) =\infty$, the assertion of the lemma is trivial. If $\W_p(\sigma_N,\sigma) < \infty$, let $\gamma\in \Gamma(\sigma_N,\sigma)$ be an optimal transport plan. 

We fix $1\leq i \leq N$, and estimate 
\begin{align}
    \frac 1 N \sum_{\{j:d_{ij}>\delta\}} d_{ij}^{-\beta}.
\end{align}
 For convenience, we set $X_i=0$ without loss of generality. For $r > 0$ to be chosen later, we split the relevant support of $\gamma$ into different parts.  
\begin{align}
    E&\coloneqq \R^d\setminus B_{\delta}(0)\times \R^d = E_1 \cup E_2 \cup E_3, \\
     E_1 &\coloneqq \biggl\{(x,y) \in E: |x| \geq \frac 1 2 |y| \biggr\}, \\
     E_2 &\coloneqq \biggl\{(x,y) \in E: r\le |x| \le \frac 1 2 |y| \biggr\},\\
    E_3 &\coloneqq \biggl\{(x,y) \in E : \delta\leq |x | \leq r \biggr\}.
\end{align}
Then,
\begin{align} \label{est:split.E}
    \frac 1 N \sum_{\{j:d_{ij}>\delta\}} d_{ij}^{-\beta} &=  \int_{E} \frac {1} {|x|^\beta} \dd \gamma(x,y) 
     =\int_{E_1\cup E_2\cup E_3} \frac {1} {|x|^\beta} \dd \gamma(x,y) .
\end{align}
\noindent\textbf{Step 2:} \emph{Estimate on $E_1$ and $E_2$.}
For all $\psi \in L^1(\R^d) \cap L^\infty(\R^d)$ and $\beta<d$ we have (see e.g. the proof of \cite[Lemma 2.4]{HoferSchubert23a})
\begin{align} \label{eq:fractional.convolution}
	 	\||\cdot|^{-\beta} \ast \psi \|_{\infty} \lesssim \|\psi\|_{\infty}^{\frac{\beta }d}\|\psi\|_{1}^{1-\frac \beta d}.
	 \end{align}
 Applying this, we deduce	 
\begin{align} \label{est:E_1}
    \int_{E_1} \frac {1} {|x|^\beta} \dd \gamma(x,y) \lesssim \int_{\R^d\times \R^d} \frac {1} {|y|^\beta} \dd \gamma(x,y) \lesssim \|\sigma\|_{\infty}^{\frac{\beta}d}.
\end{align}
We observe that in $E_2$ we have 
\begin{align}
|x - y| \geq |y| - |x| \geq \frac 12 |y|\ge r.
\end{align}
Thus,
\begin{align} \label{est:E_3}
    \int_{E_2} \frac {1} {|x|^\beta} \dd \gamma(x,y) \lesssim  r^{-(\beta+p)} \int_{E_2} |x - y|^p \dd \gamma(x,y) \lesssim   r^{-(\beta+p)} \eta^p.
\end{align}

\noindent\textbf{Step 3:} \emph{Estimate on $E_3$.}
Let 
$\J\coloneqq\{j: \delta \leq |X_j| \leq r\}$ and set 
\begin{align}
	 \psi = \frac{1}{N\delta^d\abs{B_1(0)}} \sum_{j \in \J} \1_{B_{\delta}(X_j)}. 
\end{align} 
It holds that $|X_j|\ge \frac 12\abs{y}$ for all $j \in \J$ and $y\in B_{\delta}(X_j)$, and thanks to \eqref{eq:fractional.convolution} we obtain
	 \begin{align} \label{est:E_2.0}
  \begin{aligned}
	 \int_{E_3} \frac {1} {|x|^\beta} \dd \gamma(x,y) &\leq \frac 1 N \sum_{j \in \J} \frac 1 {|X_j|^{\beta}} \ls \||\cdot|^{-\beta} \ast \psi \|_{\infty}
	 \lesssim M^{\beta/d}N^{-\beta/d}\delta^{-\beta}\left(N^{-1} |\J|\right)^{1-\beta/d}.
  \end{aligned}
	 \end{align}
  We claim that
\begin{align} \label{est:|J|}
   \frac{|\J|}{N} \lesssim r^{d} \|\sigma\|_{\infty} + \|\sigma\|_{\infty}^{\frac{p}{p+d}} \mathcal \eta^{\frac{dp}{p+d}}.
\end{align}
Inserting \eqref{est:|J|} into \eqref{est:E_2.0}  yields
\begin{align} \label{est:E_2}
	 \int_{E_3} \frac {1} {|x|^\beta} \dd \gamma(x,y) \lesssim \frac{M^{\beta/d}}{N^{\beta/d} \delta^{\beta}} \left(r^{d} \|\sigma\|_{\infty} + \|\sigma\|^{\frac{p}{d+p}}_{\infty} \mathcal \eta^{\frac{dp}{d+p}}\right)^{\frac {d-\beta}{d}}.
\end{align}
The remainder of this step is devoted to the proof of \eqref{est:|J|}. Here we use that the optimal plan $\gamma$ is actually given as the pushforward $(T,\Id)_\#\sigma$ by an optimal transport map $T$ (with $T_\#\sigma=\sigma_N$) since $\gamma$ is absolutely continuous with respect to the Lebesgue measure (cf. \cite{Santambrogio15}). Note that then 
\begin{align*}
	\eta=\bra{\int_{\R^d
    } |Ty-y|^p \dd \sigma(y)}^{1/p}.
\end{align*}
We introduce the set 
\begin{align}
    A \coloneqq \left\{|y|\ge 2r:  |Ty|<r\right\}
\end{align} 
and observe
\begin{align} \label{est:|J|.0} 
  \sigma(A) &= \sigma_N(T(A))=\sigma_N(B_r(0)\setminus T(B_{2r}(0)))=\sigma_N(B_r(0)) - \sigma(B_{2r}(0))\\
  &\geq \frac{|\J|}{N} - |B_{2r}(0)|\|\sigma\|_{\infty}.
\end{align}
In particular, this implies \eqref{est:|J|} if
$\sigma(A) \leq |\J|/(2N)$. 
It therefore remains to show \eqref{est:|J|} in the case $\sigma(A) \geq |\J|/(2N)$. We introduce 
\begin{align}
    \tilde \sigma := \sigma \1_A.
\end{align}
Moreover, we fix $\theta >0$ which satisfies
\begin{align}
    \frac{(\theta p)'} \theta < d,
\end{align}
where $(\theta p)'$ is the H\"older dual of $\theta p$. Note that the left-hand side converges to $0$ as $\theta \to \infty$ and thus such a $\theta$ exists.
We then estimate
\begin{align}
  \|\tilde \sigma\|_1 \leq \int_{\R^d} \frac{|y|^{\frac 1 \theta}}{|y|^{\frac 1 \theta}}\dd \tilde \sigma 
    &\lesssim  \left(\int_{\R^d} \frac{1}{|y|^{\frac {(\theta p )'} \theta}}\dd \tilde \sigma \right)^{\frac 1 {(\theta p)'}} 
    \left(\int_{\R^d} |y|^{p} \dd \tilde \sigma \right)^{\frac 1 {\theta p}} \\
    & \lesssim \|\tilde \sigma\|_\infty^{\frac{1}{\theta d}} \|\tilde \sigma\|_1^{\frac {d - \frac {(\theta p)'}{\theta} }{d(\theta p)'}} \left(\int_{\R^d} |y|^{p} \dd \tilde \sigma \right)^{\frac 1 {\theta p}}.
\end{align}
Thus, using that
\begin{align}
    \theta \left( 1 - \frac {d - \frac {(\theta p)' }{\theta} }{d(\theta p)' }\right) = \frac 1 p + \frac {1}{d},
\end{align}
we obtain
\begin{align}
  \eta \geq    \left(\int_{A} |Ty - y|^{p} \dd \sigma \right)^{\frac 1 { p}} \gtrsim \left(\int_{\R^d} |y|^{p} \dd \tilde \sigma \right)^{\frac 1 { p}}
  \geq \|\tilde \sigma\|_\infty^{-\frac{1}{d}} \|\tilde \sigma\|_1^{\frac 1 p + \frac {1}{d}} \gtrsim  \| \sigma\|_\infty^{-\frac{1}{d}} \left( \frac{|\J|}{N} \right)^{\frac 1 p + \frac {1}{d}},
\end{align}
which implies \eqref{est:|J|}.

\noindent\textbf{Step 4:} \emph{Conclusion.}
Inserting the estimates \eqref{est:E_1}, \eqref{est:E_3} and \eqref{est:E_2} into \eqref{est:split.E} yields
\begin{align}
     \frac 1 N \sum_{\{j:d_{ij}>\delta\}} d_{ij}^{-\beta} \lesssim \|\sigma\|_{\infty}^{\frac{\beta }d} +\frac{M^{\beta/d}}{N^{\beta/d} \delta^{\beta}} \left(r^{d} \|\sigma\|_{\infty} + \|\sigma\|_\infty^{\frac{p}{d+p}} \mathcal \eta^{\frac{dp}{d+p}}\right)^{\frac {d-\beta}{d}} +  r^{-(\beta+p)} \eta^p.
\end{align}
Choosing  
\begin{align}
    r = \left(\frac{N^{\frac\beta d} \delta^\beta \eta^p}{\|\sigma\|_{\infty}^{\frac{d-\beta}d}M^{\beta/d}} \right)^\frac{1}{d+p}
\end{align}
concludes the proof.
\end{proof}

\subsection{Main proof}

\label{sec:proof.main}

\begin{proof}[Proof of Theorem \ref{th:deterministic.new}]
We start with the observation that, since $\rho$ satisfies a conservation law, we have $\norm{\rho(t)}_1= \norm{\rho^0}_1$ for all $t>0$ and $\rho(t)\in \mP(\R^d)$ for all $t>0$. Moreover, since $\dv K=0$, equation \eqref{eq:macro} can be rewritten as the transport equation $\partial_t\rho+(K\ast\rho)\cdot\nabla \rho$. Therefore $\norm{\rho(t)}_\infty\le \norm{\rho^0}_\infty$ for all $t>0$. We will frequently use that $\norm{\rho}_\infty+\norm{\rho}_1\ls 1$. Notice furthermore that 
\begin{align} \label{dmin.upper}
	\dmin(0)\ls N^{-1/d}.
\end{align}
Indeed,  as \eqref{ass:conv} implies $\rho_N(0)\rightharpoonup \rho^0$, testing this convergence with the characteristic function of a unit cube that has non-zero intersection with the support of $\rho_0$ yields asymptotically at least $cN$ particles in that cube for some $c>0$. To avoid technical doubling, for the rest of the proof, we assume $p<\infty$. The case $p=\infty$ works with exactly the same ideas.\\

\noindent\textbf{Step 1:} \emph{Setup of the bootstrap argument.} 
Let $L_1,L_2>0$ and define
\begin{align} \label{delta(t)}
    \delta(t) \coloneqq \frac 1 2 \delta_N e^{-2L_1t}.
\end{align}
We introduce the sets
\begin{align}
    D_t := D_{\delta(t)} =  \{X_i(t) : 1 \leq i \leq N, d_{i,i_{nn}} <\delta(t) \}, \\
    G_t := \{X_i(t) : 1 \leq i \leq N, d_{i,i_{nn}} \geq \delta(t) \}.
\end{align}
Consider $T_\ast = T_\ast(N) \leq T$, the largest time for which
\begin{align}
    \sup_{0 \leq t \leq T_\ast} S_{\alpha+1,\delta(t)} &\leq N L_2  \label{S_alpha.bootstrap}, \\
    \forall i \neq j: ~ d_{ij}(t) &\geq \frac 1 2 e^{-2L_1t}d_{ij}(0). \label{d_ij.bootstrap}
\end{align}
We notice that  \eqref{d_ij.bootstrap} and
 \eqref{delta(t)} imply for all $t \in [0,T_\ast]$ that
 \begin{align} \label{dminone.delta}
     \dminone(t) \geq \delta(t). 
 \end{align}
 
In Steps 2 and 3 below, we will show that there exist $C > 0$  and $N_0 \in \N$ (that a priori depend on $L_1,L_2$) such that for all $N \geq N_0$ and all $0 \leq t \leq T_\ast$ it holds that
\begin{align}
    \W_p^p(t) \coloneqq\W_p^p(\rho_N(t),\rho(t)) & \leq  \left( \W_p^p(0) + C \rho_N^0(D_0) N^{-p} \dmin^{-\alpha p}(0)\right) e^{C (1+L_2) t},\qquad \label{W_p.bootstrap.final} \\
    \forall i \neq j: ~  d_{ij}(t)e^{L_1t} &\geq d_{ij}(0), \label{d_ij.bootstrap.final}
\end{align}
provided that $L_1$ is chosen sufficiently large. Notice that for $p=\infty$, inequality \eqref{W_p.bootstrap.final} becomes
\begin{align*}
    \W_\infty(\rho_N(t),\rho(t)) & \leq  \left( \W_\infty(0) + C N^{-1} \dmin^{-\alpha }(0)\right) e^{C (1+L_2) t}.
\end{align*}

To conclude the proof, it then suffices to show that for  $L_2>0$ sufficiently large,  $T_\ast(N)  = T$ for $N$ sufficiently large, which we prove now:
Indeed, by \eqref{dminone.delta}, we can apply Lemma \ref{lem:sums.Wasserstein} on $[0,T_\ast]$ with $\delta(t)$, $M=2$, yielding
\begin{align}
    \frac 1N S_{\alpha+1,\delta(t)} \leq C_0 \|\rho^0\|_\infty^{(\alpha+1)/d} + R_N(t),
\end{align}
where we used that $\|\rho(t)\|_\infty = \|\rho^0\|_\infty$ for all $t \geq 0$ and where $R_N(t)$ is the explicit remainder term in \eqref{eq:sums.Wasserstein.p}, i.e.
\begin{align}
 R_N \lesssim \left(\frac{1}{N^{(\alpha+1)/d} \delta^{\alpha+1}}+\left(\frac{1}{N^{(\alpha+1)/d} \delta^{\alpha+1}}\right)^{\frac{\alpha+1+p}{d+p}}\right) (\W_p(\rho_N^0,\rho^0))^{\frac{(d-\alpha - 1)p}{d+p}}.
\end{align}
By \eqref{W_p.bootstrap.final}, \eqref{delta(t)}, \eqref{ass:W_p}, and \eqref{ass:conv}, we have $R_N(t) \to 0$, uniformly on $[0,T_\ast]$, where we used that $\alpha< d-1$ and thus the second term in the bracket with exponent $\frac{\alpha+1+p}{d+p}<1$ can be bounded by the first plus $1$. 
We choose $L_2:= 2 C_0 \|\rho^0\|_\infty^{(\alpha+1)/d} $. Together with \eqref{d_ij.bootstrap.final}, a standard continuity argument shows that 
 $T_\ast(N) \geq T$ for $N$ sufficiently large.
\\

\noindent\textbf{Step 2:} \emph{Estimate for the Wasserstein distance.}
Let $\gamma_0 \in \Gamma(\rho_N^0,\rho^0)$ be an optimal transport plan such that
\begin{align}
    \W_p(0)= \bra{\int_{\R^d\times \R^d} |x-y|^p \dd \gamma_0(x,y)}^{1/p}.
\end{align}
Both $\rho_N$ and $\rho$ can be written as 
\begin{align}
	\rho(t) &= Y(t,0,\cdot)\#\rho^0,\\
    \rho_N(t) &= Y_N(t,0,\cdot)\#\rho_N^0,
\end{align}
where the second equation is meaningful for a $\rho_N$-almost everywhere defined map $Y_N$. Here, $Y_N,Y$ are the flow maps defined as
\begin{align}
	\partial_t Y(t,s,x) &= u(t,Y(t,s,x)),& ~ Y(t,t,x) &= x,\\
    \partial_t Y_N(t,s,x) &= u_N(t,Y_N(t,s,x)), & ~ Y_N(t,t,x) &= x,
\end{align}
where $u=K\ast\rho$ and $u_N(X_i)=\frac 1 N \sum_{j \neq i} K(X_i - X_j)$. For the microscopic system this is immediate from \eqref{eq:micro}, for the macroscopic system this follows by the method of characteristics. The flow maps exist as long as the underlying systems are well-posed.
Consider the push-forward transport plan $\gamma_t:=(Y_N(0,t,\cdot), Y(0,t,\cdot))_\# \gamma_0$. Then $\gamma_t \in \Gamma(\rho_N(t),\rho(t))$.  We define 
\begin{align}\label{eq:fdef}
	\eta(t) \coloneqq  \int_{\R^d\times \R^d} \abs{x-y}^p \dd \gamma_t(x,y).
\end{align}
Note that $\W_p^p(t)\le \eta(t)$. With the convention $K(0)=0$, it holds that
\begin{align}
\begin{aligned}
		\frac{\dd }{\dd t} \eta(t)
  &=  \frac{\dd }{\dd t}  \int_{\R^d \times \R^d} |Y_N(t,0,x)) - Y(t,0,y)|^p \dd \gamma_0(x,y) \\
  &\leq p \int_{\R^d \times \R^d} |Y_N(t,0,x)) - Y(t,0,y)|^{p-1}\int_{\R^d \times \R^d} |K(Y_N(t,0,x) - z) - K(Y(t,0,y)-w)|\\
  &\hspace{11.5cm}\dd \gamma_t(z,w) \dd \gamma_0(x,y) \\
  & \leq p \int_{\R^d \times \R^d} |x - y|^{p-1} \int_{\R^d \times \R^d} |K(x - z) - K(y-w)| \dd \gamma_t(z,w) \dd \gamma_t(x,y).
  \end{aligned}\label{eq:f_est}
\end{align}	
We split the integral
\begin{align} \label{split.dot.eta}
\begin{aligned}
    &\int_{\R^d \times \R^d} |x - y|^{p-1} \int_{\R^d \times \R^d} |K(x - z) - K(y-w)| \dd \gamma_t(z,w) \dd \gamma_t(x,y) \\
    & \leq \int_{G_t \times \R^d} |x - y|^{p-1} \int_{\R^d \times \R^d} |K(x - z) - K(y-w)| \dd \gamma_t(z,w) \dd \gamma_t(x,y) \\
&+\int_{D_t \times \R^d} |x - y|^{p-1} \int_{(B_{\delta(t)}(x))^c \times \R^d} |K(x - z) - K(y-w)| \dd \gamma_t(z,w) \dd \gamma_t(x,y)
    \\
    &+ \int_{D_t \times \R^d} |x - y|^{p-1}  \int_{B_{\delta(t)}(x) \times \R^d} |K(x - z) - K(y-w)| \dd \gamma_t(z,w) \dd \gamma_t(x,y).
    \end{aligned}
    \end{align} 
    Since $K$ satisfies \eqref{eq:C_alpha}, we have the estimate
\begin{align}
	|K(x) - K(y)| \leq C |x-y| \left( \frac 1 {|x|^{\alpha+1}} + \frac 1 {|y|^{\alpha+1}} \right). \label{eq:K.Lipschitz} 
\end{align}
which implies
\begin{align}
\begin{aligned}
    &\abs{K(x-z)-K(y-w)}\ls\bra{\frac{\1_{\{x \neq z\}}}{\abs{x-z}^{\alpha+1}}+\frac{1}{\abs{y-w}^{\alpha+1}}}\bra{\abs{x-y}+\abs{z-w}}.
    \end{aligned}\label{eq:integrand_est}
    \end{align}
    Since $\rho(t) \in \mP \cap L^\infty$, we have by \eqref{eq:fractional.convolution}
    \begin{align}
       \sup_{y\in \R^d} \int_{\R^d \times \R^d} \frac 1 {|y-w|^{\alpha+1}} \dd \gamma_t(z,w) = \sup_{y\in \R^d} \int_{\R^d } \frac 1 {|y-w|^{\alpha+1}}  \rho(t,w) \dd w \lesssim 1. 
    \end{align}
    Moreover,  by \eqref{S_alpha.bootstrap}
    \begin{align}
        \sup_{x\in G_t} \int_{\R^d \times \R^d} \frac 1 {|x-z|^{\alpha+1}} \dd \gamma(z,w) \leq \frac 1 N S_{\alpha+1,\delta(t)} \leq L_2, \\
        \sup_{x\in D_t} \int_{(B_{\delta(t)}(x))^c \times \R^d} \frac 1 {|x-z|^{\alpha+1}} \dd \gamma(z,w) \leq \frac 1 N S_{\alpha+1,\delta(t)} \leq L_2.
    \end{align}
Hence, by Young's inequality and Fubini's theorem
\begin{align} \label{dot.eta.good}
\begin{aligned}
    & \int_{G_t \times \R^d} |x - y|^{p-1} \int_{\R^d \times \R^d} |K(x - z) - K(y-w)| \dd \gamma_t(z,w) \dd \gamma_t(x,y) \\
    &+ \int_{D_t \times \R^d} |x - y|^{p-1} \int_{(B_{\delta(t)}(x))^c \times \R^d} |K(x - z) - K(y-w)| \dd \gamma_t(z,w) \dd \gamma_t(x,y) \\
    & \lesssim (1+L_2) \eta(t).
\end{aligned}
\end{align}

For the last integral in \eqref{split.dot.eta}, we observe that \eqref{dminone.delta} implies
that if $X_i = x$ for some $1 \leq i \leq N$, then  there exists at most one $j \neq i$ such that  $X_j \in B_{\delta(t)}(x)$.
Hence, for $N$ sufficiently large and using \eqref{d_ij.bootstrap}
\begin{align}
    \sup_{x \in D_t}  \int_{B_{\delta(t)}(x) \times \R^d} |K(x - z)| \dd \gamma_t(z,w) \leq N^{-1} (\dmin(t))^{-\alpha} \leq Ce^{Ct} N^{-1} (\dmin(0))^{-\alpha},
\end{align}
where $C$ depends on $L_1$ and $\alpha$.
Moreover, (since there is at most one $j \neq i$ such that  $X_j \in B_{\delta(t)(X_i)}$ ), we have 
$\rho_N(B_{\delta(t)}(X_i)) \leq 2/N$. Hence, the measure $\sigma_{x}$ defined by $\sigma_x(A) = \gamma(B_{\delta(t)}(x) \times A)$ satisfies
$\|\sigma\|_\infty \leq \|\rho\|_\infty$ and $\|\sigma\|_1 \leq 2/N$.
Therefore, by \eqref{eq:fractional.convolution},
\begin{align}
     \sup_{y \in \R^d} \sup_{x \in D_t} \int_{B_{\delta(t)}(x) \times \R^d} |K(y - w)| \dd \gamma_t(z,w) &= \sup_{y \in \R^d} \sup_{x \in D_t} \int_{ \R^d} |K(y - w)| \dd \sigma_x(w)
     \\
     &\leq \|\sigma\|_1^{\frac{d -\alpha }{d}} \|\sigma\|_\infty^{\frac{\alpha}{d}} \lesssim N^{-\frac{d -\alpha }{d}} .
\end{align}
Moreover, we observe that \eqref{d_ij.bootstrap} implies that 
\begin{align}
    X_i(0) \in G_0 \Longrightarrow X_i(t) \in G(t)
\end{align} 
and hence 
\begin{align}
    \rho_N(D_t) \leq \rho_N^0(D_0)  .
\end{align}
Combining these inequalities and using Young's inequality yields
\begin{align}\label{dot.eta.bad}
\begin{aligned}
     &\int_{D_t \times \R^d} |x - y|^{p-1} \int_{B_{\delta(t)}(x)\times \R^d} |K(x - z) - K(y-w)| \dd \gamma_t(z,w) \dd \gamma_t(x,y) \\
    &\lesssim  \left(N^{-1} \dmin^{-\alpha} + N^{-\frac{d -\alpha }{d}}\right)   \int_{D_t \times \R^d} |x - y|^{p-1} \dd \gamma_t(x,y) \\
    &\lesssim \eta + \rho_N(D_t) \left(N^{-1} \dmin^{-\alpha}(0) + N^{-\frac{d -\alpha }{d}}\right)^p \lesssim \eta + \rho_N^0(D_0) \left(N^{-1} \dmin^{-\alpha}(0) + N^{-\frac{d -\alpha }{d}}\right)^p .
\end{aligned}
\end{align}
Combining \eqref{eq:f_est}, \eqref{split.dot.eta}, \eqref{dot.eta.good}, \eqref{dmin.upper} and \eqref{dot.eta.bad}, we obtain
\begin{align}
    \frac{\dd }{\dd t} \eta \leq C\left((1+L_2) \eta +\rho_N(D_0) N^{-p} \dmin^{-\alpha p}(0) \right). 
\end{align}
We deduce from Gronwall's inequality that
\begin{align}
    \eta(t) \leq  \left( \eta(0) + C \rho_N(D_0) N^{-p} \dmin^{-\alpha p}(0)\right) e^{C (1+L_2) t}
\end{align}
and thus
\begin{align}
    \W_p^p(t) \leq  \left( \W_p^p(0) + C \rho_N(D_0) N^{-p} \dmin^{-\alpha p}(0)\right) e^{C (1+L_2) t}.
\end{align}
 as claimed in \eqref{W_p.bootstrap.final}. \\
 Note that in the case $p=\infty$, the proof becomes simpler since instead of the integral over $(x,y)$, there is a supremum.\\
    
\noindent\textbf{Step 3:} \emph{Estimate for the interparticle distances.}
We distinguish between close pairs of particles and sufficiently separated ones.
Consider first $i \neq j$ such that $d_{ij}(0) \leq \delta(0)$.
Then  \eqref{delta(t)}, \eqref{d_ij.bootstrap} and \eqref{ass:dminone.strong.1} imply  for all $N$ sufficiently large and  for all $k \notin \{i,j\}$ 
\begin{align} \label{delta.threshold}
    d_{ik}(t) \geq \frac 1 2 e^{-2L_1t}  d_{ik}(0)  \geq \delta(t)
\end{align}
for $N$ sufficiently large.
Similarly,
\begin{align} 
    d_{jk}(t) \geq \delta(t).
\end{align}
Note that in particular $i$ and $j$ are nearest neighbors i.e. $i = j_{nn}(0)$ and vice versa.
Hence, using that $K$ satisfies \eqref{ass:nonattractive}, we deduce from \eqref{eq:K.Lipschitz} that
\begin{align}
		\frac{\dd}{\dd t} d_{ij}^2 &=\frac 2 N (X_i - X_j) \bra{\sum_{k \neq i} K(X_i - X_k)  - \sum_{k \neq j} K(X_j - X_k)} \\	
		&\gtrsim - d^2_{ij} \frac 1 N  \sum_{k \not \in \{i,j\}}  
		\left( \frac{1}{|X_i - X_k|^{\alpha+1}} +  \frac{1}{|X_j - X_k|^{\alpha+1}} 	\right) \\
  & \gtrsim -  d_{ij}^2 \frac 1 N S_{\alpha+1,\delta(t)} \gtrsim - L_2 d_{ij}^2 .
	\end{align}

Now consider $i \neq j$ such that $d_{ij}(0) \geq \delta(0)$. Then, by \eqref{dminone.delta}, there exists at most one particle $k_i \neq i$ and $k_j \neq j$, respectively, for which there exists $t \in [0,T_\ast]$ such that
\begin{align}
    d_{ik_i} < \delta(t), && d_{jk_j} < \delta(t).
\end{align}
 Necessarily (if such particles exist) $k_i = i_{nn} \neq j$ and $k_j = j_{nn} \neq i$ .

We thus use \eqref{eq:K.Lipschitz} and \eqref{eq:C_alpha} to estimate 
	\begin{align}
		\frac{\dd}{\dd t} d_{ij}^2 &=\frac 2 N (X_i - X_j) \bra{\sum_{k \neq i} K(X_i - X_k)  - \sum_{k \neq j} K(X_j - X_k)} \\		
		&\gtrsim - d_{ij}^2 \left( \frac 1 N S_{\alpha+1,\delta(t)} + d_{ij}^{-1} \frac 1 N ( d_{ii_{nn}}^{-\alpha} + d_{jj_{nn}}^{-\alpha}) \right).
	\end{align}	 
Using \eqref{ass:dminone.strong.2} and \eqref{d_ij.bootstrap} yields
\begin{align}
      d_{ij}^{-1}(t) \frac 1 N ( d_{ii_{nn}}(t)^{-\alpha} + d_{jj_{nn}}(t)^{-\alpha}) 
       \ll 1.
\end{align}
In particular, we can choose $N$ sufficiently large such that
\begin{align} \label{est.d_ij.final}
    \frac{\dd}{\dd t} d_{ij}^2 \gtrsim  - d_{ij}^2 (1 + L_2).
\end{align}
We conclude that \eqref{est.d_ij.final} holds for all $i\neq j$, which implies \eqref{d_ij.bootstrap.final} by Gronwall's inequality upon choosing $L_1$ sufficiently large.
\end{proof}
\section{Proof of the probabilistic result}  \label{sec:prob}

Throughout this section, we consider $N$ i.i.d. particles according to some density $\rho  \in \mP(\R^d)  \cap L^{\infty}(\R^d)$. We continue to denote $\rho_N$ the empirical density of these particles.

\subsection{Preliminary probabilistic estimates}

In order to prove Theorem~\ref{th:main}, we first  establish statements that help to verify the assumptions of Theorem~\ref{th:deterministic.new}.

\begin{lem}\label{lem:Wass_scaling}
    Let $d\ge 2$ and $\Phi \in W^{1,\infty}(\R^d,\R^d)$ and let $\rho = \Phi_\# \1_{Q}$, where $Q=[0,1]^d$.
    Then,  for all $p \in [1,\infty]$ there exists $C<\infty$ depending only on $p$ and the Lipschitz constant of $\Phi$ such that
    \begin{align}
        \lim_{N\to \infty} \P\left(\W_p(\rho,\rho_N) \geq C N^{-1/d}\log N\right) = 0.
    \end{align}
\end{lem}
\begin{proof}
    For $d\ge 3$ and $p<\infty$, this is a straightforward consequence of the result \cite[Theorem 1.1 a)]{Tal94}, even without the logarithm. The theorem covers the case $\Phi=\Id$. Otherwise it is clear that the costs are comparable via the Lipschitz-map $\Phi$. The case $d=2$ and $p<\infty$ is treated in \cite{AKT84} for optimal matchings (two empirical densities for the same law), which according to \cite{GTS15} implies the present statement.  Finally, the case $p=\infty$ is treated in \cite{GTS15} (with slightly better exponents for the logarithm). For even finer results on the concentration we refer to \cite{AST19,GD21}
\end{proof}

The next lemma is a standard result that can for example be found in \cite[Proposition A.3]{Hauray09}.
\begin{lem}\label{lem:prob.distances}
There exists $C>0$ depending only on $\rho$ such that for all $L > 0$ it holds that 
	\begin{align}
		\P(\dmin \leq L^{-1} N^{-2/d}) \leq 1 - e^{-\frac{C}{L^d}} \leq \frac C {L^d}.
	\end{align}
\end{lem}

\begin{lem}\label{lem:three_particle1}
   There exists $C>0$ depending only on $\rho$ such that for all $L_1,L_2 > 0$ it holds that 
    \begin{align}
		\P(\exists i \neq j \neq k \neq i : d_{ij} \leq L_1, d_{ik} \leq L_2) \leq C N^3 L_1^d L_2^d.
	\end{align}
In particular,
 \begin{align}
     \lim_{L \to \infty} \P(\dminone \leq L^{-1} N^{-3/(2d)}) \to 0, \qquad \text{uniformly in $N$}.  
 \end{align}
\end{lem}
\begin{proof}
    We estimate
    \begin{align}
        \P(d_{12} \leq L) \leq  \sup_{x \in \R^d} \rho(B_L(x))\lesssim  L^d \|\rho\|_\infty.
    \end{align}
    Hence,
    \begin{align}
        \P(\exists i \neq j \neq k \neq i : d_{ij} \leq L_1, d_{ik} \leq L_2) &\leq N^3 \P(d_{12} \leq L_1, d_{13} \leq L_2)\\
        &\leq N^3 \sup_{x\in \R^d} \P(d_{12} \leq L_1, d_{13} \leq L_2 | X_1=x)\\
        &\leq N^3 \sup_{x\in \R^d} \P(d_{12} \leq L_1 | X_1=x) \P(d_{13} \leq L_2 | X_1=x)\\
        &\lesssim N^3 L_1^d L_2^d.
     \end{align}
\end{proof}

\begin{lem}\label{lem:three_particle2}
    Let $\beta > 0$. Then, there exists $C>0$ depending only on $\beta$, $\rho$ and $d$ such that for $\eps\in (0,1)$ and any $\delta>0$, the following estimate holds. 
    \begin{align} \label{est:three_particle2}
        &\P(\exists i \neq j \neq k \neq i : d_{ik} < \delta \wedge N^{-1} d_{ij}^{-1} d_{ik}^{-\beta} \geq N^{-\eps}) 
        \\
        &\leq C N^{-d\eps} 
        + C N^{3-d(1-\eps)} \begin{cases}
            \delta^{d(1-\beta)} &\quad \text{if } \beta < 1,\\
            \log N + |\log \delta| & \quad \text{if } \beta = 1, \\
            N^{(-2/d - \eps)d(1-\beta)}& \quad \text{if } \beta > 1.
        \end{cases}
    \end{align}
\end{lem}
\begin{proof}
    We estimate
    \begin{align}
         & \P(\exists i \neq j \neq k \neq i : d_{ik} < \delta \wedge N^{-1} d_{ij}^{-1} d_{ik}^{-\beta} \geq N^{-\eps})  \\
         &\leq \P(\dmin \leq N^{-2/d - \eps}) +  N^3 \P\Bigl( N^{-1} d_{12}^{-\beta} d_{13}^{-1} \geq N^{-\eps} \wedge N^{-2/d - \eps} \leq d_{12} \leq \delta \Bigr). \qquad  \label{est:three_particle2.1}
    \end{align}
    By Lemma \ref{lem:prob.distances}, the first right-hand side term above is estimated by the first right-hand side of \eqref{est:three_particle2}.

    We introduce the probability measure $\nu$ on $\R_+$ as the pushforward measure under $d_{12}$, i.e.
    \begin{align}
        \nu((0,a)) = \P(d_{12}  \leq a) = \int_{\R^d \times \R^d} \1_{|x_1-x_2| \leq a} \dd \rho(x_1) \dd \rho(x_2).
    \end{align}
    We notice that the coarea formula implies that $\nu$ has a density such that
    \begin{align}
        \dd \nu(a) \lesssim \|\rho\|_\infty a^{d-1} \dd a \lesssim a^{d-1} \dd a. 
    \end{align}
    Moreover, we have
    \begin{align}
        \P(d_{13} \leq  d_{12}^{-\beta} N^{-1 +\eps} | d_{12}) \lesssim \|\rho\|_\infty (N^{-1 +  \eps} d_{12}^{-\beta})^d.
    \end{align}
    Thus,
    \begin{align}
        &\P\Bigl( N^{-1} d_{13}^{-1} d_{12}^{-\beta} \geq N^{-\eps} \wedge N^{-2/d - \eps} \leq d_{12} \leq \delta \Bigr) \\ &=\int_{\{N^{-2/d - \eps} \leq d_{12} \leq \delta \}} \P(d_{13} \leq  d_{12}^{-\beta} N^{-1 +\eps} | d_{12}) \dd \P \\
        &\lesssim  \int_{N^{-2/d - \eps}}^{\delta}   (N^{-1 +  \eps}  a^{-\beta})^d \dd \nu(a) \\
         &\lesssim  N^{-d(1-\eps)} \int_{N^{-2/d - \eps}}^{\delta}    a^{d(1- \beta) - 1} \dd a \\
        &\lesssim N^{-d(1-\eps)} \begin{cases}
            \delta^{d(1-\beta)} &\quad \text{if } \beta < 1,\\
            \log N + |\log \delta| & \quad \text{if } \beta = 1, \\
            N^{(-2/d - \eps)d(1-\beta)}& \quad \text{if } \beta > 1.
        \end{cases}
    \end{align}
    Thus, the second right-hand side term in \eqref{est:three_particle2.1} is estimated by the second right-hand side term in \eqref{est:three_particle2}.
\end{proof}

\begin{lem}\label{lem:number_close_pairs}
    There exists $C > 0$ depending only on $\rho$ such that the following holds for all $\theta\in (0,1)$ and all $\delta > 0$ such that $C\delta^d<\theta$.
    \begin{align}
        \P(\#\{ i: d_{ii_{nn}} \leq \delta N^{-1/d}\} \geq 2\theta N   ) \leq 2\left(\frac 1{1-\theta}\left( \frac {C(1-\theta) \delta^d}\theta \right)^{\theta}\right)^N.
    \end{align}
\end{lem}
\begin{proof}
    We consider
    \begin{align}
        Z_i &:= \begin{cases}
            1 & \quad \text{if } \exists j < i : d_{ij} \leq \delta N^{-1/d}, \\
            0 & \quad \text{otherwise },
        \end{cases} \\
        S_N &:= \sum_{i=1}^N Z_i.
    \end{align}
    Moreover, let $\tilde Z_i, \tilde S_N$ be defined in the same way but with $i < j$ instead of $j < i$ in the definition of $\tilde Z_i$.
    Then, we have
    \begin{align}
        \P(\#\{ i: d_{ii_{nn}} \leq \delta N^{-1/d}\} \geq 2 N \theta  ) 
        &\leq \P\left(\frac 1 N (S_N + \tilde S_N) \geq 2 \theta  \right) \\
        & \leq \P\left(\frac 1 N S_N  \geq  \theta  \right) + \P\left(\frac 1 N  \tilde S_N \geq \theta  \right) = 2  \P\left(\frac 1 N S_N  \geq  \theta  \right), 
    \end{align}
  where we used in the last line that $S_N$ and $\tilde S_N$ are identically distributed.
The right-hand side can be estimated similarly to Cram\'er's Theorem for large deviations.  
    The random variables $Z_i$ are not i.i.d.
    However, denoting by $\sigma_k$ the sigma algebra generated by $(X_1, \dots, X_k )$, we have for all $\lambda \in \R$ 
    \begin{align}
        \E \left( \prod_{i=1}^N e^{\lambda Z_i} \right) &= \E(e^{\lambda Z_1} \E(e^{\lambda Z_2} \dots \E(e^{\lambda Z_{N-1}} \E(e^{\lambda Z_N} | \sigma_{N-1}) | \sigma_{N-2}) \dots |\sigma_2)| \sigma_{1})), 
    \end{align}
    and since 
    \begin{align}
        \P(Z_{i+1} = 1 | \sigma_i) \lesssim i (\delta N^{-1/d})^d \lesssim \delta^d,
    \end{align}
    we find with $q = C \delta^d$ where $C$ is the implicit constant above
    \begin{align}
          \E(e^{\lambda Z_{i+1}} | \sigma_i) & \leq q e^\lambda + (1-q).
    \end{align}
    Hence 
    \begin{align}
      \E(e^{\lambda S_N}) =  \E \left( \prod_{i=1}^N e^{\lambda Z_i} \right) \leq (q e^\lambda + (1-q))^N.
    \end{align} 
    We estimate by Markov's inequality
    \begin{align}
        \P\left(\frac 1 N S_N \geq \theta\right) \leq  \P\left(e^{\lambda S_N} \geq  e^{\lambda N\theta} \right) \leq e^{-\lambda N \theta} \E(e^{\lambda S_N}).
    \end{align}
    Thus, we get for all $\lambda \in \R$ 
    \begin{align}
      \frac 1 N \log  \P\left(\frac 1 N S_N \geq \theta\right)  \leq -\lambda \theta +  \log(q e^\lambda + (1-q)).
    \end{align}   
    Optimising in $\lambda$, i.e. choosing $\lambda = \log(\theta/q) - \log((1-\theta)/(1-q))$ yields
    \begin{align}
       \frac 1 N \log  \P\left(\frac 1 N S_N \geq \theta\right) \leq - \theta \log \left( \frac \theta q \right) - (1-\theta) \log \left(\frac {1-\theta}{1-q} \right).
    \end{align}
    In particular, using $q < \theta < 1 $,
    \begin{align}
         \P\left(\frac 1 N S_N \geq \theta\right) \leq \left(\frac 1{1-\theta}\left( \frac {C(1-\theta) \delta^d}\theta \right)^{\theta}\right)^N.
    \end{align}
    
\end{proof}

\subsection{Proof of Theorem~\ref{th:main}}

Theorem \ref{th:main} is an immediate consequence of Theorem~\ref{th:deterministic.new} and the following probabilistic result.

\begin{prop} \label{pro:assumptions.prob}
     Let $d \geq 2$ and let $\rho^0\in \mP(\R^d)\cap L^\infty(\R^d)$ be the pushforward by a Lipschitz map of the uniform  distribution on the cube $\1_Q$.  Let the family $\rho_N^0$ be empirical measures of $N$ i.i.d. particles  with law $\rho^0$.
    Then, for $0<\alpha<1/3$ if $d=2$, and $0<\alpha<\frac {d-1}2$ if $d\ge 3$, and $p \in [1,\infty]$ such that
        \begin{align*}
        p>\frac{d(\alpha+1)}{2d-3(\alpha+1)},
    \end{align*}
    there exists $\delta_N>0$ such that, with overwhelming probability, $\delta_N\le \dminone(0)$, as well as assumptions \eqref{ass:conv}--\eqref{ass:dminone.strong.2} are satisfied, and
    \begin{align} \label{absorbable} 
    N^{-1}(\rho_N^0(D_{\delta_N}))^{1/p} \dmin(0)^{-\alpha} \leq \W_p(\rho_N^0,\rho^0).
    \end{align}
    More precisely, let $A$ be the set of particle configurations for which $\delta_N\le \dminone(0)$, \eqref{ass:conv}--\eqref{ass:dminone.strong.2}, and \eqref{absorbable} hold. Then $\lim_{N \to \infty} \P(A) = 1$.
\end{prop}

\begin{proof}
    The appropriate choice is $\delta_N=N^{-3/(2d)-\eps}$ for $\eps$ small enough.  
    We have to check that  each of the conditions \eqref{ass:conv}--\eqref{ass:dminone.strong.2} is satisfied with overwhelming probability in the limit $N\to \infty$. Lemma~\ref{lem:Wass_scaling} yields \eqref{ass:conv} since most configurations satisfy $\W_p(\rho,\rho_N)\le CN^{-1/d}\log N$. Lemma~\ref{lem:three_particle1} assures $\delta_N\le \dminone(0)$ and  \eqref{ass:dminone.strong.1} for most configurations. Assumption \eqref{ass:dminone.strong.2} is satisfied due to Lemma \ref{lem:three_particle2}: We choose $\delta=\delta_N$. Then, for $d=2$, $\alpha<1/3$ and for $d=3$, $\alpha < 1$ such that the total power of $N$ on the right-hand side of \eqref{est:three_particle2} is always negative for $\eps$ small enough. For $d\ge 4$ the power before the bracket on the right-hand side of \eqref{est:three_particle2}  is negative and majorizes possible positive powers for $\alpha\ge 1$ precisely if $0<\alpha<\frac {d-1}2$. \\
    Finally, comparing the two terms appearing in \eqref{ass:W_p}, inequality \eqref{absorbable}, which we prove below, implies that it is enough to check the first term. For this, it is enough to use again Lemma~\ref{lem:Wass_scaling} to see the scaling of the Wasserstein distance as $N^{-1/d}\log N$. Using $\alpha<1/3$ for $d=2$ and $\alpha < \frac {d-1}2 \leq \frac 2 3 d -1$ for $d\ge 3$, this determines the bound $p>\frac{d(\alpha+1)}{2d-3(\alpha+1)}$ for the first term to go to zero.  
    Finally, regarding condition \eqref{absorbable}, we use Lemma~\ref{lem:prob.distances} with $L = N^{\eps}$ and Lemma~\ref{lem:number_close_pairs} with $\delta=N^{-1/(2d)-\eps}$ and $\theta=2C N^{-1/2-\eps d}$ to estimate $\rho_N^0(D_{\delta_N})$ and $ \dmin(0)$, yielding 
\begin{align*}
    N^{-1}(\rho_N^0(D_{\delta_N}))^{1/p} \dmin(0)^{-\alpha} \leq N^{\frac{-2p(d-2\alpha)-d}{2dp} + \alpha \eps}
\end{align*}
with overwhelming probability. 
In view of Remark \ref{rem:det} \ref{it:Wass_scaling}, it is enough to check 
$\frac{-2p(d-2\alpha)-d}{2dp} < -\frac 1 d$ which is always true for $ \alpha < (d-1)/2$.
\end{proof}

\section*{Acknowledgements}

R.S. has been supported by the Deutsche Forschungsgemeinschaft (DFG, German Research Foundation) through the research training group ''Energy,
Entropy, and Dissipative Dynamics (EDDy)`` (Project-ID 320021702 /GRK2326) and the collaborative research
centre ‘The mathematics of emergent effects’ (CRC 1060, Project-ID 211504053).
R.H. thanks Barbara
Niethammer,  Juan Vel'azquez and the Hausdorff Center for Mathematics for the hospitality during the stays in
Bonn.

\appendix

 \begin{refcontext}[sorting=nyt]
\printbibliography
 \end{refcontext}
\end{document}